\documentclass{amsart}
\usepackage{amssymb}
\usepackage{mathtools}
\newtheorem{theorem}{Theorem}[section]
\newtheorem{lemma}[theorem]{Lemma}
\newtheorem{proposition}[theorem]{Proposition}
\theoremstyle{definition}
\newtheorem{definition}[theorem]{Definition}

\theoremstyle{remark}

\numberwithin{equation}{section}
\DeclareUnicodeCharacter{0301}{\'{e}}

\begin{document}
\title{Urysohn 1-width for 4 and 5 manifolds with positive biRicci curvature}

\author{Junyu Ma}
\address{University of Pennsylvania}
\curraddr{University of Pennsylvania}
\email{junyuma@sas.upenn.edu}
\thanks{}

\subjclass[2023]{53C12}


\dedicatory{}

\keywords{Differential geometry, Fundamental group}

\begin{abstract}
 In this paper, we show that closed four and five manifolds with positive biRicci curvature have finite Urysohn 1-width bounded by a constant that depends on the curvature bounds. During the proof, we can also observe that the fundamental group of those manifolds is virtually free. This gives a new proof that $T^{2}\times S^{2}$ can't admit positive biRicci curvature. If we additionally assume the manifold is complete non-compact but with non-negative Ricci curvature, we can prove the linear volume growth for such manifolds.
\end{abstract}

\maketitle

\section{Introduction}

While a positive lower bound on the Ricci curvature can be used to bound the diameter of a Riemannian manifold by Bonnet-Myers, the same is not true for scalar curvature starting in dimensions $n\geq 3$, as observed by Gromov-Lawson \cite{GL80}. Nonetheless, they showed the following result in \cite{gl83}:

\begin{theorem}\label{thmgl}(Gromov-Lawson \cite{gl83})
    Suppose that $(M^{3},g)$ is complete simply connected with lower scalar curvature bound $R\geq 2$. Then for $p\in M$ with $f(x)=d_{g}(x,p)$. Then each component of $f^{-1}(t)$ has diameter bounded by $\frac{12}{\sqrt{3}}\pi$.
\end{theorem}

This implies that a certain intermediate diameter quantity called Urysohn 1-width can be bounded from scalar curvature. More recently, Liokumovich-Maximo \cite {Yd20} proved the following conjecture of Gromov:

\begin{theorem}(Liokumovich-Maximo \cite{Yd20})
    Let $(M^{3},g)$ be a compact manifold with positive scalar curvature $R_{h}\geq\Lambda_{0}>0$.Then there exists a Morse function $f:M\to \mathbb{R}$, such that for every $x\in \mathbb{R}$ and each connected component $\Sigma$ of $f^{-1}(x)$ we have:
\begin{itemize}
    \item[(a)] $Area(\Sigma_{x})\leq \frac{112\pi}{\Lambda_{0}}$.
    \item[(b)] $diam(\Sigma_{x})\leq \sqrt{\frac{2}{3}}\frac{26\pi}{\sqrt{\Lambda_{0}}}$.

    \item[(c)] $\Sigma_{x}$ has genus at most 2. 
\end{itemize}

\end{theorem}
This result was further extended to complete non-compact 3-manifolds by Lioku\-movich-Wang \cite{lw23}. It is therefore natural to consider  similar questions in higher dimensions.

Since the standard product metric on $S^{2}\times \mathbb{R}^{2}$ does not have finite 1-width, it is not possible to hope to control the 1-width of 4-manifolds from a lower bound on scalar curvature alone. In Hirsch-Brendle-Johne\cite{ssa22}, the authors focused on intermediate curvatures, which inspired us to consider this problem with biRicci curvature bound in this work. 

One important method was proved in recent work of Chodosh-Li-Liokumovich \cite{cll23} :
\begin{lemma}(Chodosh-Li-Liokumovich\cite{cll23})
Assume that $(N^{n},g)$ has the property that any embedded (n-2)-submanifold in the universal cover $\Sigma_{n-2}\in \Tilde{N}$ can be filled in $B_{L}(\Sigma_{n-2})$. Then the universal cover $(\Tilde{N},\Tilde{g})$ has the property that:\\
For any $p\in \Tilde{N}$, each component of the level set of $d(p,\cdot)$ has diameter $\leq 20L$.
\end{lemma}

They showed that in some sufficiently connected 4,5-dimensional manifolds the above assumption holds for filling radius and thus those manifolds have finite 1-width. In addition, in \cite{x23}, Xu showed that:
\begin{theorem} (Xu \cite{x23})
    Let $M$ be a complete simply-connected manifold with dimension $n \leq5$ and bi-Ricci curvature $biRic(M) \geq \lambda > 0$. Then M has finite Urysohn 1-width.
\end{theorem}

The biRicci curvature condition is an intermediate curvature assumption between scalar and Ricci. It was used on recent work of Chodosh-Li-Minter-Stryker \cite{ccpd24} to classify complete stable minimal immersions in $\mathbb{R}^5$.

Our first goal is to prove a Urysohn 1-width bound for biRicci curvature in the vein of Theorem \ref{thmgl}. We obtain:

\begin{theorem}
    For $n=4,5$, if M is a closed n-manifold that admits a metric with biRicci curvature $biRic(M)\geq \lambda$, then, there is a metric graph $(K,d)$ and a distance non-increasing map $\phi :(M^{n},g)\rightarrow (K,d)$ so that $diam(\phi^{-1}(p))\leq c(n,\lambda)$.
\end{theorem}  
A quick remark about the above theorem is that the map to the graph arises from the distance function. As a consequence, we can show the following about the fundamental groups for such manifolds:
\begin{theorem}
    For $n=4,5$, if M is a closed n-manifold that admits a metric with positive biRicci curvature then the fundamental group, $\pi_{1}(M)$, is virtually free.
\end{theorem}
By applying theorem 1.6 to 4-manifold $T^{2}\times S^{2}$, we can give another quick proof of the proposition:
\begin{proposition}
    (Hirsch-Brendle-Johne\cite{ssa22})
    $T^{2}\times S^{2}$ does not admit a metric with positive biRicci curvature.
\end{proposition}
In the remaining part of the paper, we investigate the volume growth of manifolds with a uniform bound on BiRicci curvature. For this part, we assume also that the metric has non-negative Ricci curvature outside geodesic ball of radius $a$. Using a similar trick as in Chodosh-Li-Stryker(\cite{ccd24}), we can prove that such manifolds have at most linear volume growth:
 
\begin{theorem}
    For $n=4,5$, if $M$ is a simply connected complete n-manifold that admits a metric with biRicci curvature $biRic(M)\geq \lambda>0$ and nonnegative Ricci curvautre outside geodesic ball of radius $a$, then $M$ is either compact, or $0<c_{1}\leq \liminf \limits_{r\to \infty}\frac{Vol(B_{r}(p))}{r}\leq\limsup\limits_{r\to \infty}\frac{Vol(B_{r}(p))}{r}\leq c_{2}$.\\
\end{theorem}
As a spoil of the proof, we can establish a uniform upper bound of $c_{2}$, but it is unknown to the author if we can find universal lower bound for $c_{1}$. In this direction, it is interesting to note a recent result by Zhu\cite{z24}, where he generalized the Calabi-Yau theorem of sublinear volume growth manifold manifolds with non-negative Ricci curvature outside compact sets to complete manifolds with non-negative scalar curvature in certain dimensions.

While preparing this manuscript, the author learned that Antonelli-Xu \cite{gk24} independently obtained the linear volume growth estimate, among many other results.

\subsection*{Acknowledgements}The author would like to thank Davi Maximo for multiple valuable suggestions and improvements for this work. 

\section{Filling radius bound, Urysohn 1-width and fundamental group}
In this section we first recall the work of Shen-Ye \cite{sr96} of the filling radius estimate for closed n-manifold that admits a metric with biRicci curvature bounded below by positive number. Recall the the biRicci curvature $biRicci(u,v)=Ric(u)+Ric(v)-K(u,v)$, where $K(u,v)$ is the sectional curvature of $u,v$. The key idea is to use the positive function of elliptic operator given by strictly stable minimal surface with boundary to construct a weighted functional and thus bounds the filling radius. If we let $q=-|A|^{2}-Ric(\nu,\nu)$, then we can apply work by Schoen of elliptic operator. We say a hypersurface with positive first eigenvalue of stability operator is strictly stable. Consequently we have:
\begin{proposition}
    If an embedded minimal hypersurface $S\subset M$ is strictly stable then $\exists$ a positive function $f$ satisfying:
\begin{equation}
    \Delta_{s}f+|A|^{2}f+Ric(\nu,\nu)f=0
\end{equation}
\end{proposition}
We then constructed a weighted functional\\
\begin{equation}
    I(c)=\int_{0}^{l}f|\dot{c}|dt
\end{equation}
for parameterized curve $c$ on $S$ with $\dot{c}=\frac{dc}{dt}$, $l$ the length of curve.
\begin{theorem}(Shen-Ye\cite{sr96})
  For $n=4,5$, suppose M is a (closed) n-manifold that admits a metric with biRicci curvature $biRic(M)\geq \lambda$. If an embedded minimal hypersurface $S\subset M$ is strictly stable, and $c$ is a minimizer of $I$, then the length of $c$, $l\leq c(n)$. Where $c(4)=\sqrt{\frac{2}{\lambda}}\pi, c(5)=\sqrt{\frac{4}{\lambda}}\pi$.
 \end{theorem}
The detailed proof is given in Shen-Ye\cite{sr96} and the key idea is to calculate the first and second variation of the weighted functional $I$ defined above. \\

Shen and Ye defined the homology radius of a codimension-2 embedded submanifold $\Gamma$ as(which is same as filling radius we discussed above):
\begin{definition}
\begin{equation}
 r(\Gamma)=\sup\left\{r>0, \Gamma \text{ is not homologous to 0 in r-neighborhood of }\Gamma \right\}
\end{equation}
\end{definition}
The homology radius of manifold $M$ is defined as:
\begin{definition}

\begin{equation}
\begin{aligned}
r(M)=\sup\left\{r(\Gamma), \Gamma\subset M \text{is an embedded} \right.\\
\left. \text{ submanifold of codimension-2 homologous to 0 in }M \right\}
 \end{aligned}
\end{equation}  
\end{definition}
\begin{theorem}(Shen-Ye\cite{sr96})
    For $n=4,5$, if M is a closed n-manifold that admits a metric with biRicci curvature $biRic(M)\geq \lambda$. Then if $S$ is an embedded stable minimal hypersurface, then the diameter of $S$, $diam(S)\leq c(n)$. Where $c(4)=\sqrt{\frac{2}{\lambda}}\pi, c(5)=\sqrt{\frac{4}{\lambda}}\pi .$
\end{theorem}
Here we quickly summarize the proof by Shen-Ye. 
\begin{proof}
Suppose $p,q\in S$, then by taking of geodesic balls with radius $\delta$ and centers $p,q$ we obtain a strictly stable hypersurface $S_{\delta}$. Construct the positive function $f$ given in proposition 2.1 and then define the functional $I$. Take a minimizer among all curves connecting $\partial B_{\delta}(p)$ and $\partial B_{\delta}(q)$, $c$. Then by theorem 2.2 we have $l(c)\leq c(n)$. Then by triangle inequality we have $d(p,q)\leq 2\delta+c(n) $. By taking $\delta\rightarrow 0$ we finished the proof. 
\end{proof}

Then for any codimension-2 cycle $\Gamma=\partial S$ of 4 and 5 dimensional manifolds, $S$ can be presented by embedded minimal hypersurfaces and thus we can bound the filling radius.
\begin{theorem}(Shen-Ye\cite{sr96})
    For $n=4,5$, if $M$ is a closed n-manifold that admits a metric with biRicci curvature $biRic(M)\geq \lambda$, then the filling radius of $M$, $r(M)\leq c(n)$. Where $c(4)=\sqrt{\frac{2}{\lambda}}\pi, c(5)=\sqrt{\frac{4}{\lambda}}\pi .$
\end{theorem}

  Recall the lemma of Chodosh-Li-Liokumovich\cite{cll23}, they provided a method of proving Urysohn 1-width with filling radius bound. 

\begin{lemma}(Chodosh-Li-Liokumovich\cite{cll23})
Assume that $(N^{n},g)$ has the property that any embedded $(n-2)$-submanifold in the universal cover $\Sigma_{n-2}\in \Tilde{N}$ can be filled in $B_{L}(\Sigma_{n-2})$. Then the universal cover $(\Tilde{N},\Tilde{g})$ has the property that:\\
For any $p\in \Tilde{N}$, each component of the level set of $d(p,\cdot)$ has diameter $\leq 20L$.
\end{lemma}

Mimicing the proof of above lemma provided by Chodosh, Li and Liokumovich, we can prove the theorem 1.5:

\begin{proof}
Suppose M is a closed n-manifold that admits a metric with biRicci curvature $biRic(M)\geq \lambda$. Then for the induced metric on $\Tilde{M}$ such that covering map $\pi$ is local isometry, $biRic(\Tilde{M})\geq \lambda$. Therefore by theorem 2.5 the filling radius of $\Tilde{M}$ is bounded by $L$. Let $p\in \Tilde{M}$ be a point such that $f(q)=d(p,q)$ violates the conclusion above. Then we can find $x,y\in f^{-1}(t)\subset \Tilde{M}$ with $d(x,y)>20L$, where $L$ is the filling radius bound proved in above section. Connect x,y,p with minimizing geodesics $\gamma,\eta_{x},\eta_{y}$, where $\eta_{x}$ goes from $p$ to $x$, $\gamma$ goes from $x$ to $y$ and $\eta_{x}$ goes from $y$ to $p$. Let $T$ be the triangle $\eta_{x}\star\gamma\star\eta_{y}$. Fix $0<l<L$ such that $\partial B_{4L+l}(x)$ and $\partial B_{L+\epsilon}(\eta_{x})$ are smooth hypersurfaces intersecting transversely(Perturbations of metric balls might be needed to make boundary smooth). Define $\Sigma_{n-2}=\partial B_{4L+l}(x) \cap \partial B_{L+\epsilon}(\eta_{x})$, and $\Sigma_{n-1}=\partial B_{4L+l}(x) \cap \overline{B_{L+\epsilon}}$. Notice $\partial \Sigma_{n-1}=\Sigma_{n-2}$, and thus $\Sigma_{n-2}$ is actually a boundary and can be applied filling radius estimate.\\
If $\Sigma_{n-2}=\varnothing$ then $d(\Sigma_{n-2},\cdot)=\infty$ in the following discussion. Then by construction we have $d(\Sigma_{n-2},\eta_{x})>L$ and since $\eta_{x}$ is minimizing geodesic, it intersects transversely with $\Sigma_{n-1}$ exactly once. We have following distance estimate:\\
\begin{equation}
    d(\Sigma_{n-2},\gamma)\geq(\Sigma_{n-1},\gamma)>L
\end{equation}
\begin{equation}
    \eta_{y}\cap \Sigma_{n-1}=\emptyset
\end{equation}
\begin{equation}
    d(\Sigma_{n-2},\eta_{y})>L
\end{equation}

To conclude what we got above, we have an one cycle $T$ consists of geodesics, a submanifold $\Sigma_{n-2}$ whose distance to each component of $T$ is strictly bigger than $L$, and a hypersurface $\Sigma_{n-1}$ with boundary $\Sigma_{n-2}$ and intersect $T$ transversely exactly once. 

Now we are going to finish the proof with simply connectedness of $\Tilde{M}$. Perturb the triangle $T$ to $T'$ to be a smoothly embedded 1-cycle intersects transversely $\Sigma_{n-1}$ exactly once.(It holds for $T$ and thus true for small perturbation). If $\Sigma_{n-2}\neq \varnothing$ then since it is boundary of $\Sigma_{n-1}$ and by assumption of filling radius, we can find a $\Sigma_{n-1}'\subset B_{L}(\Sigma_{n-2})$. We have $dist(\Sigma_{2},T)>L$, and thus $\Sigma_{n-1}'\cap T=\emptyset$. Therefore if we perturb $T$ small we have $\Sigma_{n-1}'\cap T'=\emptyset$. Then $T'$ has nontrivial algebraic intersection with the cycle $\Sigma_{n-1}-\Sigma_{n-1}'$ contradicts the fact that $\Tilde{N}$ is simply connected. If $\Sigma_{n-2}=\emptyset$, then $\Sigma_{n-1}$ itself is a cycle with nontrivial algebraic intersection with $T'$. Which is still a contradiction. We finish proving the 1-width of the universal cover of the manifolds. There is only 1 small lemma to be done.
\end{proof}

\begin{lemma}
    Suppose $\Tilde{N}$ is the universal cover of $N$ with induced metric having the property that $p\in \Tilde{N}$, each component of the level set of $d(p,\cdot)$ has diameter $\leq C$, then $N$ also has this property.
\end{lemma}
\begin{proof}
Suppose there is $p,x,y\in N$ with $d(x,p)=d(x,y)=t$ in the same connected component of $d^{-1}(t)$ and $d(x,y)>C$. Then let $l_{x}$ be the minimizing geodesic connecting $p,x$ and $l_{y}$ connecting $p,y$. Find $\Tilde{{p}}$ with $\pi(\Tilde{{p}})=p$. Since the lift of minimizing geodesic is minimizing geodesic, there exists $\Tilde{{x}}$ with $\pi(\Tilde{{x}})=x$ and geodesic $\Tilde{{l}}_{x}$ connecting $\Tilde{{p}}$ and $\Tilde{{x}}$ with same length of $l_{x}$. \\
Also for any preimage of $y$ under covering map, $\Tilde{y}$, by the covering map and similar discussion above we have $d(x,y)\leq d(\Tilde{x},\Tilde{y})$. Which gives contradiction if we can prove that $\Tilde{x}$ and $\Tilde{y}$ are in the same connected component of $\Tilde{d}^{-1}(t)$. Take a compact neighborhood $N'$ of $p,x,y$, then for any deck transformation that is not identity $\sigma$, $q\in \Tilde{N'}$, $\Tilde{d}(q,\sigma(q))\geq C_{0}>0$. Therefore if we let $U=\left\{z\in N:d(p,z)=t\right\}$, for any $z_{0}\in N'$, $\exists r_{0}>0$ such that if $d(z_{1},z_{0})<r_{0}$, $z_{1}\in U$ then we can lift to have $\Tilde{z_{0}}$, $\Tilde{z_{0}}$, with $\Tilde{d}(\Tilde{z_{0}},\Tilde{p})=\Tilde{d}(\Tilde{z_{1}},\Tilde{p})=t$. Thus we can lift a curve in $U$ connecting $x$ and $y$, which implies $\Tilde{x}$ and $\Tilde{y}$ are in the same connected component of $\Tilde{d}^{-1}(t)$. 
\end{proof}

In Ramachandran-Wolfson\cite{mw09}, one important use of fill radius of curves bound is to show that any finitely generated subgroup $G$ of $\pi_{1}(N)$ can't have one end, which can prove the fundamental group is virtually free. Here we applied their method directly to conclude the case for closed manifolds with positive biRicci curvature. 

\begin{theorem}(Ramachandran-Wolfson\cite{mw09})
    Let N be a closed Riemannian manifold. Suppose that the universal cover $\pi:\Tilde{N}\rightarrow N$ is given the metric $\Tilde{g}$ such that $\pi$ is a local isometry. If $(\Tilde{N},\Tilde{g})$ has fill radius above then the fundamental group $\pi_{1}(N)$ is virtually free. 
\end{theorem}

\begin{lemma}(Chodosh-Li-Liokumovich \cite{cll23})
    If N is a closed manifold whose universal cover has bounded Urysohn 1-width, then any finitely generated subgroup $G$ of fundamental group $\pi_{1}(N)$ can't have one end. 
\end{lemma}
Combining lemma 2.10 and theorem 1.5 we have the following proposition. 
\begin{proposition}
 
    For $n=4,5$, if M is a closed n-manifold that admits a metric with positive biRicci curvature then any finitely generated subgroup $G$ of fundamental group $\pi_{1}(M)$ can't have one end. 
\end{proposition}

Combining proposition 2.11 and the proof used by Ramachandran-Wolfson, we finished the proof for theorem 1.6.\\

In Hirsch-Brendle-Johne\cite{ssa22} about generalized Geroch conjecture, they provided a proof of proposition 1.7. Here by theorem 1.6 we can give another proof:
\begin{proof}
 Since $\pi_{1}(T^{2}\times S^{2})=\mathbb{Z}\times \mathbb{Z}$, which is abelian. Therefore the only free subgroup is infinitely cyclic.  However any such group will not be of finite index. By theorem 1.5, $T^{2}\times S^{2}$ does not admit a positive biRicci curvature.
\end{proof}

\section{linear volume growth for non-negative ricci curvature and positive bi-ricci curvature}
Notice that we don't use the compactness in proving theorem 1.5, which is necessary in showing the fundamental group. Two key statement used in proof, theorem 2.6 and lemma 2.7 hold when the manifold is complete with uniform positive biRicci curvature lower bound. Therefore, with similar proof above we claim the following property:\\

\begin{proposition}
    For $n=4,5$, if M is a simply connected complete n-manifold that admits a metric with biRicci curvature $biRic(M)\geq \lambda>0$.Then, there is a metric graph $(K,d)$ and a distance non-increasing map $\phi :(M^{n},g)\rightarrow (K,d)$ so that $diam(\phi^{-1}(p))\leq c(n,\lambda)$.
\end{proposition}

 We denote $c(n,\lambda)$ as $L$ in the following discussion. We first consider the volume growth of ends of such manifolds. Here we use the definition of end in Chodosh-Li-Stryker(\cite{ccd24}).
\begin{definition}
    A collection of open sets $\left\{E_{k}\right\}$ is an end adapted to $x$ with length scale $L$ if each $E_{k}$ is an unbounded connected component of $M\backslash \bar{B}_{kL}(x)$ and satisfies $E_{k+1}\subset E_{k}$.
\end{definition}
We listed the topological proposition that will be needed as well:\\
\begin{proposition}(Chodosh-Li-Stryker\cite{ccd24})
    If $M$ is simply connected and $\left\{E_{k}\right\}$ is an end adapted to $x$ with length scale $L$, then both $\bar{E}_{k}\backslash E_{k+1}$ and $\partial E_{k}$ are connected for all $k$.
\end{proposition}
Then we conclude the main theorem for volume growth of ends:\\
\begin{theorem}
     For $n=4,5$, if $M$ is simply connected n-manifolds that admits a metric with biRicci curvature $biRic(M)\geq \lambda>0$ and nonnegative Ricci curvauture and $\left\{E_{k}\right\}$ is an end adapted to $x$ with length scale $L$, then there is a constant (possibly depending on $L$) $C>0$ such that for $M_{k}\coloneq E_{k}\cap B_{(k+1)L}(x)$, $Vol(M_{k})\leq C$.
\end{theorem}

Here we follow the strategy used in proving linear volume growth of complete stable minimal hypersurface in 4-manifolds with positive biRicci curvature in Chodosh-Li-Stryker(\cite{ccd24}). 

\begin{proof}
    Since we can generically chose the point such that the distance function to that point is morse, by perturbing small distance, we may assume that $kL+\frac{L}{2}$ is not a critical point of distance function $f(p)=d(x,p)$. Therefore by connectedness of $M$ we conclude that the level sets $\cup \Sigma_{i}$, where $\Sigma_{i}$ denotes a connected component of the sets $\left\{p|d(x,p)=kL+\frac{L}{2}\right\}$, separate $\partial E_{k}$ from $\partial E_{k+1}$. Then we claim that there exist a single level set $\Sigma_{i_{0}}$ that separate $\partial E_{k}$ from $\partial E_{k+1}$. Suppose not, then since $\cup \Sigma_{i}$ separates $\partial E_{k}$ from $\partial E_{k+1}$, we can take a curve from $\partial E_{k}$ to $\partial E_{k+1}$ in $\bar{M}_{k}$, which intersects level sets transverse and intersects some $\Sigma_{i_{0}}$ odd times. If $\Sigma_{i_{0}}$ does not separate $\partial E_{k}$ from $\partial E_{k+1}$, then we can find another curve from $\partial E_{k}$ to $\partial E_{k+1}$ in $\bar{M}_{k}\backslash \Sigma_{i_{0}}$, which gives a loop with nontrivial intersection with $\Sigma_{i_{0}}$ contradicting that $M$ is simply connected. \\
   
    Then we conclude that $M_{k}$ has bounded diameter. For any $y,z\in M_{k}$, take radial geodesics from $x$ to $y$ and $x$ to $z$, then they both intersect with $\Sigma_{i_{0}}$. The length of radial geodesics in $M_{k}$ is less or equal to $L$ by construction. By proposition 4.1 we have $diam(\Sigma_{i_{0}})\leq 20L$. Therefore $d(y,z)\leq L+20L+L=22L$.\\

Since Ricci curvature of $M$ is bounded by 0, we can take Bishop-Gromov volume comparison to conclude that $Vol(M_{k})\leq C$. 
    
\end{proof}

We nearly finished the proof of the main theorem 1.8 of this section with the fact that the manifolds with nonnegative Ricci curvature outside a compact has finitely many ends. Together with the following Calabi-Yau theorem we finished the proof of theorem 1.8.

\begin{theorem}(Calabi-Yau)
    For $n=4,5$, let $(M^{n},g)$ be a complete n-manifolds with non-negative Ricci curvature outside a compact subset. Then $(M^{n},g)$ has sublinear volume growth iff $M^{n}$ is compact. 
\end{theorem}

\begin{theorem}(Cai \cite{c91})
    If M is an n-manifold that admits non-negative Ricci curvature outside geodesic ball of radius $a$, then the number of ends of $M$ has a universal bounds $b(n,a)$. Where $b(4,a)\leq \frac{8e^{26a}}{3a^{4}}$ and $b(5,a)\leq \frac{5e^{34a}}{2a^{5}}$.
\end{theorem}

Therefore combining theorem 3.4 and theorem 3.6 we can conclude that $M$ has at most linear volume growth outside geodesic ball. And by multiplying the constants we can know that $c_{2}\leq \frac{Vol(B_{n}(22c(n)))b(n,a)}{c(n)}$, where $c(n)$ is the filling radius bound, $B_n(l)$ is the n-Euclidean ball of radius $l$ and $b(n,a)$ is the bound of ends. Notice $\lim\limits_{r\to \infty}\frac{B(x,a)}{r}=0$. In the end, together with the Calabi-Yau theorem 3.5 we finish the proof of theorem 1.8.
\\
\\
Remark: If we additionally assume the whole manifold with non-negative Ricci curvature, then by the standard splitting theorem, the manifolds can have at most 2 ends and we can take the $b(n,a)=2$ in the proof of theorem 1.8. 
\bibliographystyle{amsplain}
\bibliography{work}
\end{document}